\documentclass[11pt, a4paper, reqno, ]{amsart}
\usepackage{amscd}
\usepackage{%
amssymb,latexsym, xspace, enumerate, verbatim, }
\def\bb#1{\ifmmode {\mathbb #1}\else {$\mathbb #1$}\fi}
\font\bit=cmmib10
\newcommand{\mb}[1]{\ifmmode{\hbox{\bit #1\/}}\else{{\bit #1\/}}\fi}

\theoremstyle{plain}
\newtheorem{theo}{Theorem}
\newtheorem{cor}[theo]{Corollary}

\newtheorem{lem}[theo]{Lemma}
\newtheorem{prop}[theo]{Proposition}
\newtheorem{Def}[theo]{Definition}
\newtheorem{rem}[theo]{Remark}

\numberwithin{equation}{section}

\newcommand{\jac}[1]{\ensuremath{{\rm J}#1}}
\newcommand{\mat}[2]{\ensuremath{{\rm M}_{#1}(#2)}}

\newcommand{\Z}{\ensuremath{{\mathbb Z}}\xspace}

\begin{document}
\title{$r$-clean rings}
\author{NAHID ASHRAFI and Ebrahim nasibi}
\address{
Department of Mathematics\\ Semnan University\\ Semnan, Iran\\}%
 \email{<nashrafi@semnan.ac.ir>, <ashrafi49@Yahoo.com>}%
\address{
Department of Mathematics\\Semnan University\\ Semnan, Iran\\}%
\email{<ebrahimnasibi@yahoo.com>, <enasibi@gmail.com>}%


\begin{abstract}
An element of a ring $R$ is called clean if it is the sum of an
idempotent and a unit. A ring $R$ is called clean if each of its
element is clean. An element $r\in R$ called regular if $r=ryr$
for some $y\in R$. The ring $R$ is regular if each of its element
is regular. In this paper we define a ring is $r$-clean if each of
its elements is the sum of a regular and an idempotent element. We
give some relations between $r$-clean and clean rings. Finally we
investigate some properties of $r$-clean rings.
\end{abstract}
\maketitle

\section{Introduction}
Throughout this paper, $R$ denotes an associative ring with
identity, $U(R)$ the group of units, $Id(R)$ the set of
idempotents, $\jac (R)$ the Jacobson radical and $\mat {n}{R}$
the ring of all $n\times n$ matrices over $R$.\\
An element $x$ of a ring is called clean if $x=u+e$, where $u\in
U(R)$ and $e\in Id(R)$. A ring $R$ is called clean if each of its
element is clean. Clean rings first were introduced by Nicholson
\cite{Nic}. Several peoples worked on this subject and
investigate properties of clean rings, for example see \cite{And}, \cite{Cam}, \cite{Han1} and \cite{Nic2}.\\
A ring $R$ is said to be exchange ring if for each $a\in R$ there
exists $e\in Id(R)$ such that $e\in aR$ and $(1-e)\in (1-a)R$.
Nicholson \cite[Proposition 1.8]{Nic} proved that clean rings are
exchange and a ring with central idempotents is clean if and only
if it is exchange. Camillo and Yu \cite[Theorem 9]{Cam} proved
that a ring is semiperfect if and only if it is clean
and has no infinite orthogonal family of idempotents.\\
In 1936, von Neumann defined that an element $r\in R$ is regular
if $r=ryr$ for some $y\in R$, the ring $R$ is regular if each of
its element is regular. Some properties of regular rings has been
studied in \cite{Goo}. A ring $R$ is called unit regular if, for
each $a\in R$, there exists a unit $u\in R$ such that $aua=a$.
Camillo and Yu \cite[Theorem 5]{Cam} proved that every unit
regular ring is clean. In \cite{Nic1} Nicholson and Varadarajan
 proved that the converse is not true.\\
Let $Reg(R)=\{a\in R:$ a is regular$\}$. We call an element $x$ of
a ring $R$ is $r$-clean if $x=r+e$, where $r\in Reg(R)$ and $e\in
Id(R)$. A ring $R$ is $r$-clean if each of its element is
$r$-clean.\\
Clearly regular rings and clean rings are $r$-clean. But we will
give some examples that shows in general, $r$-clean rings may not
be regular. Also we will give an example that shows in general,
$r$-clean rings may not be clean. In fact Bergman \cite[Example
1]{Han} has constructed a regular ring which is not directly
finite ( a ring $R$ is directly finite
 if for any elements $a,b\in R$, $ab=1$ implies $ba=1$) with $2$ invertible in which not every element is a
sum of units. We will
 prove that Bergman's example is not clean, but clearly it is $r$-clean.\\
We will show that a directly finite ring $R\neq 0$ is local if and
only if it is $r$-clean and $0$ and $1$ are the only idempotents
in $R$. Finally we give some properties of $r$-clean rings and we
will prove that if $R$ is an $r$-clean ring, then so is the matrix
ring $\mat {n} {R}$ for any $n\geq 1$.

\section{$r$-clean rings}
In this section first we define $r$-clean element and $r$-clean
rings and we show that the class of clean rings are a proper
subset the class of $r$-clean rings.
\begin{Def}
An element x of a ring $R$ is $r$-clean if $x=r+e$, where $r\in
Reg(R)$ and $e\in Id(R)$, a ring $R$  is $r$-clean if each of its
element is $r$-clean.
\end{Def}
Clearly regular rings and clean rings are $r$-clean. But in
general, $r$-clean rings may not be regular. For example, every
semiperfect ring is clean (see \cite[Theorem 9]{Cam}), so it is
$r$-clean. But clearly it is not regular. Further $\Z_{4}$ is not
regular, because $\overline{2}$ is not regular in $\Z_{4}$, but it
is easy to check
that $\Z_{4}$ is $r$-clean.\\
Also in general, $r$-clean rings may not be clean. For example
\cite[Example 1]{Han}, proceeding as Bergman's example, let $F$ be
a field with $char(F)\neq 2$, $A=F[[x]]$ and $K$ be the field of
fractions of $A$. All the ideals of $A$ are generated by power of
$x$, denote by $(x^{n})$. Define:
$$R=\{r\in End (A_{F}): there~exists~q\in K~and~a~positive~integer~n,$$$$with~r(a)=qa~for~all~a\in  (x^{n}) \}.$$
By \cite[Example 1]{Han}, $R$ is a regular ring which is not
directly finite and $R$ is not generated by its units. So every
element of $R$ is not a sum of units, and since $char(F)\neq 2$,
$2$ is invertible in $R$. Also since $R$ is regular thus $R$ is
$r$-clean. But $R$ is not clean, because in otherwise, Proposition
$10$ in \cite{Cam} ( Let $R$ be a ring in which $2$ is invertible.
Then $R$ is clean if and only if every element of $R$ is the sum
of a unit and a square root of $1$) implies that every element in
$R$ is a sum of a unit and a square root of $1$. Thus every
element in $R$ is a sum of two units which it is a contradiction.\\
In following, we investigate some conditions in which $r$-clean rings are clean.\\

\begin{lem}\label{local}
A ring $R\neq 0$ is local if and only if it is clean and $0$ and
$1$ are the only idempotents in $R$.
\end{lem}
\begin{proof}
See \cite[Lemma 14]{Nic2}.
\end{proof}
\begin{theo}\label{clean}
If $R\neq 0$ is a directly finite $r$-clean ring and $0$ and $1$
are the only idempotents in $R$, then $R$ is clean.
\end{theo}
\begin{proof}
Since $R$ is $r$-clean, each $x\in R$ has the form $x=r+e$, where
$r\in Reg(R)$ and $e\in Id(R)$. If $r=0$, then $x=e=(2e-1)+(1-e)$
and since $2e-1\in U(R)$ and $1-e\in Id(R)$, so $x$ is clean.
Hence $R$ is clean. But if $r\neq0$, then there exists $y\in R$
such that $ryr=r$. Thus $ry\in Id(R)$. So by hypothesis, $ry=0$ or
$ry=1$. Now if $ry=0$, then $r=ryr=0$ which is a contradiction.
Therefore $ry=1$ and since $R$ is directly finite so $yr=ry=1$.
Thus $r\in U(R)$. So $x$ is clean and hence $R$ is clean.
\end{proof}
\begin{cor}
A directly finite ring $R\neq 0$ is local if and only if it is
$r$-clean and $0$ and $1$ are the only idempotents in $R$.
\end{cor}
\begin{theo}
Let $R$ be a commutative $r$-clean ring and each pair of
idempotent in $R$ be orthogonal. Then $R$ is clean.
\end{theo}
\begin{proof}
By \cite[Theorem 10]{And}, for commutative rings every regular
ring is clean. So for each $x\in R$, we can write
$x=e_{1}+e_{2}+u$, where $e_{1},e_{2}\in Id(R)$ and $u\in U(R)$.
Now since $e_{1}$ and $e_{2}$ are orthogonal, $e=e_{1}+e_{2}\in
Id(R)$. Hence $x=e+u$ is clean, which shows that $R$ is clean.
\end{proof}
Now, we give some properties of $r$-clean rings.\\

\begin{theo} \label{1-x}
Let $R$ be a ring, then $x\in R$ is $r$-clean if and only if $1-x$
is $r$-clean.
\end{theo}
\begin{proof}
Let $x\in R$ be $r$-clean. Then write $x=r+e$, where $r\in Reg(R)$
and $e\in Id(R)$. Thus $1-x=-r+(1-e)$. But there exists $y\in R$
such that $ryr=r$. Hence $(-r)(-y)(-r) = - (ryr) = -r$ and since
$-r\in Reg(R)$ and $1-e\in Id(R)$, so $1-x$ is $r$-clean.\\
Conversely, if $1-x$ is $r$-clean, write $1-x=r+e$, where $r\in
Reg(R)$ and $e\in Id(R)$. Thus $x=-r+(1-e)$, like previous part,
$-r\in Reg(R)$ and $1-e\in Id(R)$. Therefore $x$ is $r$-clean.
\end{proof}
\begin{cor}
Let $R$ be a ring and $x\in \jac (R)$. Then $x$ is $r$-clean.
\end{cor}
\begin{proof}
Let $x\in \jac (R)$. Then $1-x\in U(R)$. So  $1-x\in Reg(R)$.
Hence $1-x$ is $r$-clean. Therefore by Theorem \ref{1-x}, $x$ is
$r$-clean.
\end{proof}
\begin{theo}\label{factor}
Every factor ring of an $r$-clean ring is $r$-clean. In particular
a homomorphic image of an $r$-clean ring is $r$-clean.
\end{theo}
\begin{proof}
Let $R$ be $r$-clean and $I\lhd R$. Also let $\overline{x}=x+I \in
  \frac{R}{I}$. Since $R$ is $r$-clean so we have $x=r+e$, where $r\in Reg(R)$ and $e\in
  Id(R)$. Thus $\overline{x}=\overline{r}+\overline{e}$. But there exists $y\in R$ such that
  $ryr=r$. Therefore $\overline{r}\overline{y}\overline{r}=\overline{r}$. So
$\overline{r}\in Reg(R)$ and since $\overline{e}\in
Id(\frac{R}{I})$, it follows that $\frac{R}{I}$ is $r$-clean.
\end{proof}
\begin{rem}
In general, inverse of above theorem may not be correct. For
example, if $p$ be a prime number, then $\frac{\Z} {p\Z} \cong
\Z_p$ is $r$-clean, but $\Z$ is not $r$-clean.
\end{rem}
\begin{theo}\label{product}
A direct product $R=\prod_{i\in I} R_{i}$ of rings $\{R_{i}\}_{i
\in I}$ is $r$-clean if and only if so is each $\{R_{i}\}_{i\in
I}$.
\end{theo}
\begin{proof}
One direction immediately follows from Theorem \ref{factor}.\\
Conversely, let $R_{i}$ be $r$-clean for each $i\in I$. Set
$x=(x_{i})_{i \in I}~ \in \prod_{i\in I} R_{i}$. For each $i$,
write $x_{i}=r_{i}+e_{i}$, where $r_{i}\in Reg(R_{i})$ and
$e_{i}\in Id(R_{i})$. Since $r_{i} \in Reg(R_{i})$, there exists
$y_{i}\in R_{i}$ such that $r_{i}y_{i}r_{i}=r_{i}$. Thus
$x=(r_{i})_{i \in I}+(e_{i})_{i \in I}$, where $(r_{i})_{i \in
I}\in Reg(\prod_{i\in I} R_{i})$ and $(e_{i})_{i \in I}\in
Id(\prod _{i\in I}R_{i})$. Therefore $\prod _{i\in I}R_{i}$ is
$r$-clean.
\end{proof}
\begin{lem}\label{R[x]}
Let $R$ be a commutative ring and $f=\sum_{i=0} ^{n}a_{i}x^{i}\in
R[x]$ be regular. Then $a_{0}$ is regular and $a_{i}$ is nilpotent
for each $i$.
\end{lem}
\begin{proof}
Since $f$ is regular, thus there exists $g=\sum_{i=0}
^{m}b_{i}x^{i}\in R[x]$ such that $fgf=f$. So
$a_{0}b_{0}a_{0}=a_{0}$. Therefore $a_{0}$ is regular. Now to end
the proof, it is enough to show that for each prime ideal $P$ of
$R$; every $a_{i}\in P$. Since $P$ is prime, thus $\frac{R}{P}[x]$
is an integral domain. Define $\varphi:R[x]\rightarrow
\frac{R}{P}[x]$ by $\varphi(\sum_{i=0} ^{k}a_{i}x^{i})=\sum_{i=0}
^{k}(a_{i}+p)x^{i}$. Clearly $\varphi$ is an epimorphism. But
$\varphi(f)\varphi(g)\varphi(f)=\varphi(f)$, so
$deg(\varphi(f)\varphi(g)\varphi(f))=deg(\varphi(f))$. Thus
$deg(\varphi(f))+deg(\varphi(g))+deg(\varphi(f))=deg(\varphi(f))$.
Therefore $deg(\varphi(f))+deg(\varphi(g))=0$. So
$deg(\varphi(f))=0$. Thus $a_{1}+P= ... = a_{n}+P=P$, as required.
\end{proof}
\begin{theo}
If $R$ is a commutative ring, then $R[x]$ is not $r$-clean.
\end{theo}
\begin{proof}
We show that $x$ is not $r$-clean in $R[x]$. Suppose that $x=r+e$,
where $r\in Reg(R[x])$ and $e\in Id(R[x])$. Since $Id(R)=Id(R[x])$
and $x=r+e$, so $x-e$ is regular. Hence by pervious Lemma, $1$
should be nilpotent, which is a contradiction.
\end{proof}
\begin{rem}
Even if $R$ is a field, then $R[x]$ is not $r$-clean.
\end{rem}
\begin{cor}
If $R$ is a commutative ring, then $R[x]$ is neither clean nor
regular.
\end{cor}
\begin{theo}
Let $R$ be a ring. Then the ring $R[[x]]$ is $r$-clean if and only
if so is $R$.
\end{theo}
\begin{proof}
If $R[[x]]$ is $r$-clean, then by Theorem \ref{factor}, $R\cong
\frac{R[[x]]}{(x)}$ is $r$-clean.\\
Conversely, suppose that $R$ is $r$-clean. We know that
$R[[x]]\cong \{(a_{i}):a_{i}\in R$, for each $i\geq
0\}=\prod_{i\geq 0} R$. So the result is clear by Theorem
\ref{product}.
\end{proof}
\begin{theo} \label{mat}
For every ring $R$, we have the following statements:
\begin{enumerate}[\bf (1)]
\item If $e$ is an central idempotent element of $R$ and $eRe$ and
$(1-e)R(1-e)$ are both $r$-clean, then so is $R$.

\item Let $e_{1}, ...,  e_{n}$ be orthogonal central idempotents
with $e_{1}+ ...+e_{n}=1$. Then $e_{i}Re_{i}$ is $r$-clean for
each $i$, if and only if $R$ is $r$-clean.

\item If $R$ is $r$-clean, then so is the matrix ring $\mat {n}
{R}$ for any $n\geq 1$.

\item If $R$ is $r$-clean and $M$ is a free $R$-module of rank
$n$, then $End(M)$ is $r$-clean.

 \item If $A$ and $B$ are rings and $M=_{B}M_{A}$ is a bimodule, the formal triangular matrix ring
$T=\left( {\begin{array}{*{20}c}
   {A } & {0 }  \\
   {M } & {B }  \\
\end{array}} \right)$
is $r$-clean, then both $A$ and $B$ are $r$-clean.

\item Let for each integer $n\geq 2$, the ring $T$ of all $n\times
n$ lower( resp., upper) triangular matrices over $R$ be $r$-clean.
Then $R$ is $r$-clean.
\end{enumerate}
\end{theo}
\begin{proof}
We use $\bar{e}$ to denote $1-e$ and
apply the Pierce decomposition for the ring $R$, i.e.,\\
 $$R=eRe\oplus eR\bar{e}\oplus\bar{e}Re\oplus\bar{e}R\bar{e}.$$
But idempotents in $R$ are central, so $R=eRe\oplus
\bar{e}R\bar{e}\cong \left({\begin{array}{*{20}c}
   {eRe } & {0 }  \\
   {0} & {\bar{e}R\bar{e} }  \\
\end{array}} \right)$. For each $A\in R$, write  $A=\left({\begin{array}{*{20}c}
   {a } & {0 }  \\
   {0} & {b }  \\
\end{array}} \right)$, where $a,b$ belong to $eRe$ and
$\bar{e}R\bar{e}$, respectively. By our hypothesis $a,b$ are
$r$-clean. Thus $a=r_{1}+e_{1}$, $b=r_{2}+e_{2}$, where $r_{1},
r_{2}\in Reg(R)$ and $e_{1}, e_{2}\in Id(R)$. So
$$A=\left({\begin{array}{*{20}c}
   {a } & {0 }  \\
   {0} & {b }  \\
\end{array}} \right)=\left({\begin{array}{*{20}c}
   {r_{1}+e_{1} } & {0 }  \\
   {0} & {r_{2}+e_{2} }  \\
\end{array}} \right)=\left({\begin{array}{*{20}c}
   {r_{1} } & {0 }  \\
   {0} & {r_{2} }  \\
\end{array}} \right)+\left({\begin{array}{*{20}c}
   {e_{1} } & {0 }  \\
   {0} & {e_{2} }  \\
\end{array}} \right).$$
But there exists $y_{1}, y_{2}\in R$ such that
$r_{1}y_{1}r_{1}=r_{1}$, $r_{2}y_{2}r_{2}=r_{2}$. Therefore
$$\left({\begin{array}{*{20}c}
   {r_{1} } & {0 }  \\
   {0} & {r_{2} }  \\
\end{array}} \right)\left({\begin{array}{*{20}c}
   {y_{1} } & {0 }  \\
   {0} & {y_{2} }  \\
\end{array}} \right)\left({\begin{array}{*{20}c}
   {r_{1} } & {0 }  \\
   {0} & {r_{2} }  \\
\end{array}} \right)= \left({\begin{array}{*{20}c}
   {r_{1}y_{1}r_{1} } & {0 }  \\
   {0} & {r_{2}y_{2}r_{2}}  \\
\end{array}} \right)= \left({\begin{array}{*{20}c}
   {r_{1} } & {0 }  \\
   {0} & {r_{2} }  \\
\end{array}} \right).$$
So $\left({\begin{array}{*{20}c}
   {r_{1} } & {0 }  \\
   {0} & {r_{2} }  \\
\end{array}} \right)\in Reg(R)$, since $\left({\begin{array}{*{20}c}
   {e_{1} } & {0 }  \\
   {0} & {e_{2} }  \\
\end{array}} \right)\in Id(R)$, it follows that $R$ is $r$-clean.\\
On direction of $(2)$ follows from $(1)$ by induction. Conversely,
let $R$ be a $r$-clean ring and $e_{1}, ...,  e_{n}$ be orthogonal
central idempotents with $e_{1}+ ...+e_{n}=1$. Then since $R=
e_{1}Re_{1} \oplus ... \oplus e_{n}Re_{n}$ then  by Theorem
\ref{factor}, $e_{i}Re_{i}$ is $r$-clean for each $i$.\\
Also $(3)$ follows from $(2)$, and $(4)$ follows from $(3)$ and
the fact that $End(R^{n})\cong \mat {n} {R}$.\\
 For the proof of $(5)$,
let $T=\left( {\begin{array}{*{20}c}
   {A } & {0 }  \\
   {M } & {B }  \\
\end{array}} \right)$
be $r$-clean. Then for every $a\in A$, $b\in B$ and $m\in M$,
write $\left( {\begin{array}{*{20}c}
   {a } & {0 }  \\
   {m } & {b }  \\
\end{array}} \right)=\left( {\begin{array}{*{20}c}
   {f_{1} } & {0 }  \\
   {f_{2} } & {f_{3} }  \\
\end{array}} \right)+ \left( {\begin{array}{*{20}c}
   {r_{1} } & {0 }  \\
   {r_{2} } & {r_{3} }  \\
\end{array}} \right)$, where $\left( {\begin{array}{*{20}c}
   {f_{1} } & {0 }  \\
   {f_{2} } & {f_{3} }  \\
\end{array}} \right)\in Id(T)$ and $\left( {\begin{array}{*{20}c}
   {r_{1} } & {0 }  \\
   {r_{2} } & {r_{3} }  \\
\end{array}} \right)\in Reg(T)$. So $a = f_{1}+ r_{1}$ and $b = f_{3}+ r_{3}$. But there exists $\left(
{\begin{array}{*{20}c}
   {y_{1} } & {0 }  \\
   {y_{2} } & {y_{3} }  \\
\end{array}} \right)$ such that $$\left( {\begin{array}{*{20}c}
   {r_{1} } & {0 }  \\
   {r_{2} } & {r_{3} }  \\
\end{array}} \right) \left( {\begin{array}{*{20}c}
   {y_{1} } & {0 }  \\
   {y_{2} } & {y_{3} }  \\
\end{array}} \right) \left( {\begin{array}{*{20}c}
   {r_{1} } & {0 }  \\
   {r_{2} } & {r_{3} }  \\
\end{array}} \right)= \left( {\begin{array}{*{20}c}
   {r_{1} } & {0 }  \\
   {r_{2} } & {r_{3} }  \\
\end{array}} \right).$$ So
 $$\left( {\begin{array}{*{20}c}
   {r_{1}y_{1}r_{1} } & {0 }  \\
   {r_{2}y_{1}r_{1}+ r_{3}y_{2}r_{1}+r_{3}y_{3}r_{2} } & {r_{3}y_{3}r_{3} }  \\
\end{array}} \right)= \left( {\begin{array}{*{20}c}
   {r_{1} } & {0 }  \\
   {r_{2} } & {r_{3} }  \\
\end{array}} \right) .$$
Hence $r_{1}\in Reg(A)$ and $r_{3}\in Reg(B)$. It is easy to check
that $f_{1}\in Id(A)$ and $f_{3}\in Id(B)$. Therefore $a$ and $b$
are $r$-clean. Hence both $A$ and $B$ are $r$-clean.\\
To end the proof, we can see $(6)$ follows from $(5)$.
\end{proof}
\begin{prop}\label{idempotent}
Let $R$ be a $r$-clean ring and $e$ be a central idempotent in
$R$. Then $eRe$ is also $r$-clean.
\end{prop}
\begin{proof}
Since $e$ is central, it follows that $eRe$ is homomorphic image
of $R$. Hence the result follows from Theorem \ref{factor}.
\end{proof}
\begin{theo}
Let $R$ be a ring in which $2$ is invertible. Then $R$ is
$r$-clean if and only if every element of $R$ is the sum of a
regular and a square root of $1$.
\end{theo}
\begin{proof}
Suppose that $R$ is $r$-clean and $x\in R$, then $\frac{x+1}{2}\in
R$. Write $\frac{x+1}{2}=r+e$, where $r\in Reg(R)$ and $e\in
Id(R)$. So $x=(2e-1)+2r$. But there exists $y\in R$ such that
$ryr=r$. Thus $(r+r)\frac{y}{2}(r+r)=\frac{ryr}{2}+
\frac{ryr}{2}+\frac{ryr}{2}+\frac{ryr}{2}=\frac{1}{2}(r+r+r+r)=2r$.
Thus $2r\in Reg(R)$ and since $(2e-1)^{2}=1$, so $x$ is a sum of
a regular and a square root of $1$.\\
Conversely, if $x\in R$, then $2x-1=t+r$, where $t^{2}=1$ and
$r\in Reg(R)$. Thus $x=\frac{t+1}{2}+ \frac{r}{2}$. It is easy to
check that $\frac{t+1}{2}\in Id(R)$. Now  since
$\frac{r}{2}(y+y)\frac{r}{2}=\frac{ryr}{4}+
\frac{ryr}{4}=\frac{r}{2}$, it follows that $\frac{r}{2}\in
Reg(R)$. Hence $x$ is $r$-clean, which shows that $R$ is
$r$-clean.
\end{proof}
If $G$ is a group and $R$ is a ring we denote the group ring over
$R$ by $RG$. If $RG$ be $r$-clean, then $R$ is $r$-clean by
Theorem \ref{factor}. But it seems to be difficult to characterize
$R$ and $G$ for which $RG$ is $r$-clean in general. In
\cite{Chen}, \cite{Han1} and \cite{WWm} have given some rings and
groups that $RG$ is clean (so is $r$-clean). In following we will
give some rings and groups that $RG$ is $r$-clean.
\begin{theo}
Let $R$ be a commutative semiperfect ring, $G$ be a group and
$(eRe)G$ be $r$-clean for each local idempotent $e$ in $R$. Then
$RG$ is $r$-clean.
\end{theo}
\begin{proof}
Since $R$ is semiperfect, so by \cite[Theorem 27.6]{Anderson}, $R$
has a complete orthogonal set $e_{1}, ..., e_{n}$ of idempotents
with each $e_{i}Re_{i}$ a local ring for each $i$. So $e_{i}$ is a
local idempotent for each $i$. Now by hypothesis, $(e_{i}Re_{i})G$
is $r$-clean. Since $(e_{i}Re_{i})G\cong e_{i}(RG)e_{i}$ for each
$i$, it follows that $e_{i}(RG)e_{i}$ is $r$-clean. Hence $RG$ is
$r$-clean by Theorem \ref{mat} (2).
\end{proof}
\begin{theo}
Let $R$ be a ring which $2$ is invertible and $G=\{$1,g$\}$ be a
group. Then $RG$ is $r$-clean if and only if $R$ is $r$-clean.
\end{theo}
\begin{proof}
One direction is trivial.\\
Conversely, if $R$ is $r$-clean, then since $2$ is invertible by
\cite[Proposition 3]{Han1}, $RG\cong R\times R$. Hence $RG$ is
$r$-clean by Theorem \ref{product}.
\end{proof}

\end{document}